\documentclass[12pt]{amsart}
\usepackage{amssymb, amscd, amsmath, amsthm, latexsym, enumerate}
\newtheorem{theorem}{Theorem}
\newtheorem{lemma}[theorem]{Lemma}

\newtheorem*{cor}{Corollary}

\begin{document}

\title{The linking pairings of orientable Seifert manifolds}

\author{Jonathan A. Hillman }
\address{School of Mathematics and Statistics\\
     University of Sydney, NSW 2006\\
      Australia }

\email{jonathan.hillman@sydney.edu.au}

\begin{abstract}
We compute the $p$-primary components of the linking pairings of orientable
3-manifolds admitting a fixed-point free $S^1$-action.
Using this, we show that any linking pairing
on a finite abelian group with homogeneous $2$-primary summand
is realized by such a manifold.
However some pairings on inhomogeneous 2-groups are not
realizable by any Seifert fibred 3-manifold.
\end{abstract}



\maketitle

The linking pairings of oriented 3-manifolds which are Seifert fibred 
over non-orientable base orbifolds were computed in \cite{CH}.
Here we shall consider the remaining case, 
when the base orbifold is also orientable.
Thus the Seifert fibration is induced by a fixed-point free 
$S^1$-action on the manifold.
(We shall henceforth call such a space a ``Seifert manifold",
for brevity.)

Every linking pairing on a finite abelian group 
is realized by a closed orientable 3-manifold \cite{KK}.
Bryden and Deloup have used the cohomological formulation 
(recalled below) to show that if the group has odd order
the pairing is realised by a Seifert manifold
which is a $\mathbb{Q}$-homology sphere \cite{BD}.
We shall work directly with the geometric definition,
giving a new proof of this result,
and showing that there are pairings on 2-primary groups
which are not realized by any orientable Seifert fibred 3-manifold
at all (i.e., even if we allow non-orientable base orbifolds).

We give presentations for the localization of the torsion
at a prime $p$ in \S2, which lead to explicit formulae 
for the localized linking pairings in \S3.
We then study the cases $p$ odd and $p=2$ separately,
in \S\S4-5 and \S\S6-7, respectively.
Every linking pairing on a finite abelian group whose
2-primary subgroup is isomorphic to $(\mathbb{Z}/2^k\mathbb{Z})^\rho$ 
(for some $k,\rho$) is realized by some Seifert manifold.
The manifold may be either an ${\mathbb{H}^2\times\mathbb{E}^1}$-manifold,
or a $\widetilde{\mathbb{SL}}$-manifold or a $\mathbb{Q}$-homology sphere.
Theorem 10 gives some constraints on the 2-primary component of the pairing.
Pairings satisfying these constraints and  
a further ``gap" condition are realizable, by Theorem 11.
The final section \S8 summarizes briefly the earlier work of Oh on
the Witt classes of such pairings \cite{Oh}.

\section{linking pairings}

A {\it linking pairing\/} on a finite abelian group $N$ is a symmetric
bilinear function $\ell:N\times {N}\to\mathbb{Q/Z}$ which is nonsingular
in the sense that $\tilde\ell:n\mapsto\ell(-,n)$ defines an isomorphism 
from $N$ to $Hom(N,\mathbb{Q/Z})$.
If $N_1$ is a subgroup of $N$ then 
$N_1^\perp=\{n\in N\mid \ell_M(n,n_1)=0~\forall{n_1\in{N_1}}\}$.
Such a pairing is {\it metabolic\/} if there is a subgroup 
$P$ with $P=P^\perp$,
{\it split\/} \cite{KK} if also $P$ is a direct summand 
and {\it hyperbolic\/} if $N$ is the direct sum of two such subgroups.
If $\ell$ is split $N$ is a direct double.
A linking pairing $\ell$ is {\it even\/} if 
$2^{k-1}\ell(x,x)\in\mathbb{Z}$ for all $x\in{N}$ such that $2^kx=0$,
and {\it odd} otherwise.
Hyperbolic pairings are even.

Every linking pairing splits uniquely as the orthogonal sum (over primes $p$) 
of its restrictions to the $p$-primary subgroups of $N$,
and so our basic strategy is to localize at a prime.
Linking pairings on finite abelian $p$-groups may be further decomposed 
as the orthogonal sum of pairings on homogeneous summands.
If $N$ is homogeneous we shall say that $\ell$ is homogeneous,
while if $\ell=\perp_{i=1}^t\ell_i$,
where $\ell_i$ is a pairing on an homogeneous group of exponent $p^{k_i}$,
and $k_1>\dots>k_t>0$ we shall say that each $\ell_i$ is a
{\it component} of $\ell$.
When $p$ is odd the decomposition into such components is essentially unique.
In particular, every linking pairing on an abelian group of odd order
is an orthogonal sum of pairings on cyclic groups.
However, if the order is even we need also pairings 
$E_0^k$ on $(\mathbb{Z}/2^k\mathbb{Z})^2$, for each $k\geq1$,
and $E_1^k$ on $(\mathbb{Z}/2^k\mathbb{Z})^2$,
for each $k>1$.
The pairings $E_0^k$ are hyperbolic,
while the $E_1^k$s are even but not hyperbolic.
(See \cite{KK,Wa64}.

Let $\ell$ be a linking pairing on $N\cong(\mathbb{Z}/p^k\mathbb{Z})^\rho$
and $L\in\mathrm{GL}(\rho,\mathbb{Z}/p^k\mathbb{Z})$ 
be the matrix with $(i,j)$ entry $p^k\ell(e_i,e_j)$, 
where $e_1,\dots,e_\rho$ is some basis for $N$.
The {\it rank} of $\ell$ is $rk(\ell)=\mathrm{dim}_{\mathbb{F}_p}N/pN=\rho$.
If $p$ is odd then a linking pairing $\ell$ 
on a free $\mathbb{Z}/p^k\mathbb{Z}$-module $N$ is determined 
up to isomorphism by $rk(\ell)$ and the image $d(\ell)$ of $\mathrm{det}(L)$ 
in $\mathbb{F}_p^\times/(\mathbb{F}_p^\times)^2=\mathbb{Z}/2\mathbb{Z}$.
(This is independent of the choice of basis for $N$.)
If $p=2$ and $k\geq3$ then $\ell$ is determined by the image of $L$ in
$\mathrm{GL}(\rho,\mathbb{Z}/8\mathbb{Z})$;
if moreover $\ell$ is even and $k\geq2$ then $\rho$ is even 
and $\ell$ is determined by the image of $L$ in 
$\mathrm{GL}(\rho,\mathbb{Z}/4\mathbb{Z})$.
(See \cite{De,KK,Wa64}.)

If $w=\frac{p}q\in\mathbb{Q}^\times$ (where $(p,q)=1$)
let $\ell_w$ be the pairing on $\mathbb{Z}/q\mathbb{Z}$ given by 
$\ell_w(n,n')=[nn'w]\in\mathbb{Q}/\mathbb{Z}$.
Then $\ell_w\cong\ell_{w'}$ if and only if $a^2w'=b^2w$ 
for some integers $a,b$ with $(a,q)=(b,q)=1$.
In particular, if $q=2^k$ then $\ell_w\cong\ell_{w'}$ 
if and only if $2^kw'\equiv2^kw$ {\it mod} $(2^k,8)$.

If $M$ is a closed oriented 3-manifold Poincar\'e duality determines
a linking pairing $\ell_M:T(M)\times T(M)\to\mathbb{Q/Z}$, which may be 
described as follows.
Let $w$, $z$ be disjoint 1-cycles representing elements of $T(M)$ and 
suppose that $mz=\partial C$ for some nonzero $m\in\mathbb{Z}$
and some 2-chain $C$ which is transverse to $w$.
Then $\ell_M ([w],[z])=(w\bullet{C})/m\in\mathbb{Q/Z}$.
It follows easily from the Mayer-Vietoris theorem and duality that if $M$ 
embeds in $\mathbb{R}^4$ then $\ell_M$ is hyperbolic.
(If $X$ and $Y$ are the closures of the components of $\mathbb{R}^4-M$ and 
$T_X$ and $T_Y$ are the kernels of the induced homomorphisms from 
$T(M)$ to $H_1(X;\mathbb{Z})$ and $H_1(Y;\mathbb{Z})$ (respectively) then 
$T(M)\cong T_X\oplus T_Y$ and the restriction of $\ell_M$ 
to each of these summands is trivial \cite{KK}).

The linking pairing has a dual formulation, in terms of cohomology.
Let $\beta_{\mathbb{Q}/\mathbb{Z}}:H^1(M;\mathbb{Q}/\mathbb{Z})\to
{H^2(M;\mathbb{Z})}$ be the Bockstein homomorphism associated
with the coefficient sequence
\[0\to\mathbb{Z}\to\mathbb{Q}\to\mathbb{Q}/\mathbb{Z}\to0,\]
and let $D:H_1(M;\mathbb{Z})\to{H^2(M;\mathbb{Z})}$ 
be the Poincar\'e duality isomorphism.
Then $\ell_M$ may be given by the equation
\[\ell(w,z)=(D(w)\cup\beta_{\mathbb{Q}/\mathbb{Z}}^{-1}D(z))([M])
\in\mathbb{Q}/\mathbb{Z}.\]
The cup-product and Bockstein structure on $H^*(M;\mathbb{F}_p)$
for $M$ a Seifert manifold have been computed in \cite{BZ},
and this approach was used in \cite{BD} to show that pairings on
abelian groups of odd order may be realized by Seifert fibred 
$\mathbb{Q}$-homology spheres.

\section{the torsion subgroup}

Assume now that $M=M(g;S)$ is a Seifert manifold with Seifert data 
$S=((\alpha_1,\beta_1),\dots,(\alpha_r,\beta_r))$,
where $r\geq1$ and $\alpha_i>1$ for all $i\leq{r}$. 
Let $N_i$ be a torus neighborhood of the $i^{th}$ exceptional fibre,
and let $B_o$ be a section of the restriction of the
Seifert fibration to ${M\setminus\cup{intN_i}}$.
Then $B_o$ is homeomorphic to the surface of genus $g$ 
with $r$ open 2-discs deleted.
Let $\xi_i$ and $\theta_i$ be simple closed curves on $\partial{N_i}$,
corrresponding to the $i$th boundary component of $B_o$ 
and a regular fibre on $N_i$, respectively.
The fibres are naturally oriented, as orbits of the $S^1$-action.
We may assume that $B_o$ is oriented so that regular fibres 
have negative intersection with $B_o$,
and the $\xi_i$ are oriented compatibly with $\partial B_o=\Sigma\xi_i$.
Then there are 2-discs $D_i$ in $N_i$ such that 
$\partial D_i=\alpha_i \xi_i+\beta_i\theta_i$,
since $\alpha_iq_i+\beta_ih=0$ in $H_1(N_i;\mathbb{Z})$.
Hence $H_1(M;\mathbb{Z})\cong\mathbb{Z}^{2g}\oplus{H}$, 
where $H$ has a presentation
\[\langle
{q_1,\dots,q_r,h}\mid\Sigma{q_i}=0,\,\alpha_iq_i+\beta_ih=0,\,
\forall{i\geq1}\rangle.\]
(Here $q_i$ represents the image of $\xi_i$, and
$h$ represents the image of regular fibres such as $\theta_i$.)
The torsion subgroup $T(M)$ is a subgroup of $H$.
Let $\varepsilon_S=-\Sigma\frac{\beta_i}{\alpha_i}$
be the generalized Euler invariant of the Seifert fibration.

We shall modify this presentation to obtain one 
with more convenient generators.
Our approach involves localizing at a prime $p$.
After reordering the Seifert data, if necessary, 
we may assume that $\alpha_{i+1}$ divides $\alpha_i$
in $\mathbb{Z}_{(p)}$, for all $i\geq1$.
(Note that $v=\alpha_1\varepsilon_S$ is then in $\mathbb{Z}_{(p)}$,
while if $\varepsilon_S=0$ then $\frac{\alpha_1}{\alpha_2}$
is invertible in $\mathbb{Z}_{(p)}$.)
Localization loses nothing, since $\ell_M$ is uniquely 
the orthogonal sum of pairings on the $p$-primary summands of $T(M)$.
(We shall often write $\ell_M$ rather than $\mathbb{Z}_{(p)}\otimes\ell_M$,
for simplicity of notation.)

Using the relation $\Sigma{q_i}=0$ to eliminate the generator $q_1$,
we see that $\mathbb{Z}_{(p)}\otimes{H}$ has the equivalent presentation
\[\langle
{q_2,\dots,q_r,h}\mid\alpha_1\varepsilon_S{h}=0,\,\alpha_iq_i+\beta_ih=0,\,
\forall{i\geq2}\rangle.\]
If $r=1$ this group is cyclic, generated by the image of $h$.
We shall asume henceforth that $r\geq2$,
since all pairings on cyclic groups are
realizable by such Seifert manifolds.
Then there are integers
$m,n$ such that $m\alpha_2+n\beta_2=1$,
since $\alpha_2$ and $\beta_2$ are relatively prime.
Let $\gamma_i=\frac{\alpha_2}{\alpha_i}\beta_i$ and
$q_i'=\gamma_2q_i-\gamma_iq_2$, for all $i$.
(Then $q_2'=0$.)
Let $s=-mh+nq_2$ and $t=\alpha_2q_2+\beta_2h$. 
Then $h=-\alpha_2s+nt$ and $q_2=\beta_2s+mt$.
Since $t=0$ in $H$ this simplifies to
\[\langle
{q_3',\dots,q_r',s}\mid\alpha_1\alpha_2\varepsilon_Ss=0,
\,\alpha_iq_i'=0,\,\forall{i\geq3}\rangle.\]
In particular, $M(0;S)$ is a $\mathbb{Q}$-homology sphere (i.e., $H=T(M)$)
if and only if $\varepsilon_S\not=0$. 

If exactly $r_p$ of the cone point orders $\alpha_i$ 
are divisible by $p$ and $\varepsilon_S=0$ then 
$T(M)$ has nontrivial $p$-torsion if and only if $r_p\geq3$, 
in which case $\mathbb{Z}_{(p)}\otimes{T(M)}$ is the direct sum 
of $r_p-2$ nontrivial cyclic submodules,
while if $\varepsilon_S\not=0$ and $r_p\geq2$ then 
$\mathbb{Z}_{(p)}\otimes{T(M)}$ is the direct sum 
of $r_p-1$ nontrivial cyclic submodules.
(Note however that if $r_p\leq1$ and $\varepsilon_S\not=0$ then 
$\mathbb{Z}_{(p)}\otimes{T(M)}\cong
\mathbb{Z}_{(p)}/\alpha_1\varepsilon_S\mathbb{Z}_{(p)}$,
and may be non-trivial.)

\section{the linking pairing}

The Seifert structure gives natural 2-chains relating the 1-cycles
representing the generators of $H$.
We may choose disjoint annuli $A_i$ in $M^*$
with $\partial{A_i}=\theta_2-\theta_i$, for $i\not=2$.
For convenience in our formulae, we shall also let $A_2=0$.
Then 
\[C_i=\beta_2D_i-\beta_iD_2+\beta_2\beta_i{A_i}\]
is a singular 2-chain with 
$\partial{C_i}=\alpha_i\beta_2\xi_i-\alpha_2\beta_i\xi_2$.

Let $\xi_i'=\gamma_2\xi_i-\gamma_i\xi_2$, for $i\geq3$,
$\sigma=-m\theta_2+n\xi_2$ and 
\[
U=-\alpha_1B_o+\alpha_1\varepsilon_SnD_2
+\Sigma\frac{\alpha_1}{\alpha_i}(D_i+\beta_iA_i).
\]
Then $\xi_i'$ is a singular 1-chain 
representing $q_i'$ and $\partial{C_i}=\alpha_i\xi_i'$, for all $i\geq3$,
$\sigma$ is a singular 1-chain representing $s$ 
and $U$ is a singular 2-chain with 
$\partial{U}=\alpha_1\alpha_2\varepsilon_S\sigma$.

In order to calculate intersections and self-intersections 
of the 1-cycles $\xi_i$ with the 2-chains $C_i$ and $U$ in $M$, 
we may push each $\xi_i$ off $N_i$ and $B_o$.
Then $\xi_i$ and $D_j$ are disjoint, for all $i,j$,
while $\xi_2\bullet{A_i}=1$, $\xi_i\bullet{A_i}=-1$
and $\xi_j\bullet{A_i}=0$, if $i,j\not=2$ and $j\not=i$.
Similarly, we may assume that $\theta_2$ is disjoint from the discs $D_j$
(for all $j$) and the annuli $A_k$ (for all $k\not=2$).
We also have $\theta_2\bullet{B_o}=-1$, by the orientation conventions of \S2.
Hence
\[\xi_i'\bullet{C_i}=-\beta_2\beta_i(\gamma_2+\gamma_i),\]
\[\xi_i'\bullet{C_j}=-\beta_2\beta_j\gamma_i,\]
and
\[\xi_i'\bullet{U}=\alpha_1\varepsilon_S\gamma_i\]
for all $i,j\geq3$ with $j\not=i$, while
\[\sigma\bullet{C_i}=n\beta_2\beta_i\]
for all $i\geq3$ and
\[\sigma\bullet{U}=-\frac{\alpha_1}{\alpha_2}-n\alpha_1\varepsilon_S.
\]
Then
\[\ell_M(q_i',q_i')=
[-\beta_2\beta_i\frac{\alpha_i\beta_2+\alpha_2\beta_i}{\alpha_i^2}]
\in\mathbb{Q}/\mathbb{Z}\]
and
\[\ell_M(q_i',q_j')=[-\beta_2\beta_i\beta_j\frac{\alpha_2}{\alpha_i\alpha_j}]
\in\mathbb{Q}/\mathbb{Z}.\]

If $\varepsilon_S\not=0$ then the above calculations of
$\xi_i'\bullet{U}$ and $\sigma\bullet{C_i}$ each give
\[\ell_M(s,q_i')=[\frac{\beta_i}{\alpha_i}]\in\mathbb{Q}/\mathbb{Z},\]
while
\[\ell_M(s,s)=[-\frac{\alpha_1+n\alpha_1\alpha_2\varepsilon_S}
{\alpha_1\alpha_2^2\varepsilon_S}]
\in\mathbb{Q}/\mathbb{Z}.\]
In particular, the linking pairings depend only on $S$ and not on $g$.
(We could arrange that the denominators are powers of $p$,
after further rescaling the basis elements.
However that would tend to obscure the dependence on the Seifert data.)

Let $S$ and $S'$ be two systems of Seifert data, with
concatenation $S''$, and let $M''=M(g+g';S'')$
be the fibre-sum of $M=M(g;S)$ and $M'=M(g';S')$.
Then $\varepsilon_{S''}=\varepsilon_S+\varepsilon_{S'}$.
The next result is clear.

\begin{lemma}
Let $M=M(g;S)$ and $M'=M(g';S')$ be Seifert manifolds such that
all the cone point orders of $S'$ are relatively prime 
to all the cone point orders of $S$ and $\varepsilon_{S'}=\varepsilon_S=0$,
and let $M''=M(g+g';S'')$.
Then $\varepsilon_{S''}=0$ and $\ell_{M''}=\ell_M\perp\ell_{M'}$.
\qed
\end{lemma}

Thus if every $p$-primary summand of a linking pairing $\ell$ 
can be realized by some $M(0;S)$ with all cone point orders powers of $p$
and $\varepsilon_S=0$ then $\ell$ can also be realized by a Seifert manifold.

If one of the hypotheses fails,
it is not clear how the linking pairings of $M,M'$ and $M''$ are related.
In order to realize pairings by Seifert manifolds with $\varepsilon_S\not=0$
we shall need another approach.

\section{the homogeneous case: $p$ odd}

In this section we shall show that when $p$ is odd and
the $p$-primary component of $T(M)$ is homogeneous 
the structure of the $p$-primary component of $\ell_M$
may be read off the Seifert data.
Our results shall later be extended to the inhomogeneous cases.

Let $u_i=p^{-k}\alpha_i$, for $1\leq{i}\leq{r_p}$.

\begin{lemma} 
Let $M=M(g;S)$ be a Seifert manifold and $p$ a prime.
Then $\mathbb{Z}_{(p)}\otimes{T(M)}$ is homogeneous 
of exponent $p^k$ if and only if either
\begin{enumerate}
\item $\varepsilon_S=0$ and $u_i=p^{-k}\alpha_i$ is invertible 
in $\mathbb{Z}_{(p)}$, for $1\leq{i}\leq{r_p}$; or
\item $p^{-k}\alpha_1\varepsilon_S$ and $u_i$ are
invertible in $\mathbb{Z}_{(p)}$, for $2\leq{i}\leq{r_p}$; or
\item $r_p\leq2$ and $p^{-k}\alpha_1\alpha_2\varepsilon_S$ 
is invertible in $\mathbb{Z}_{(p)}$.
\end{enumerate}
\end{lemma}

\begin{proof}
This follows immediately from the calculations in \S2,
with the folowing observations.
If $\varepsilon_S=0$ then $\alpha_1$ and $\alpha_2$ must 
have the same $p$-adic valuation.
If $r_p>2$ then $u_2=p^{-k}\alpha_2$ and $v=\alpha_1\varepsilon_S$ are 
in $\mathbb{Z}_{(p)}$.
Hence if $u_2v=p^{-k}\alpha_1\alpha_2\varepsilon_S$
is invertible in $\mathbb{Z}_{(p)}$
then $u_2$ and $v$ are also invertible in $\mathbb{Z}_{(p)}$.
\end{proof}

Note that if $\mathbb{Z}_{(p)}\otimes{T(M)}$ is homogeneous 
of exponent $p^k$ and $\varepsilon_S\not=0$ then $u_1$ 
may be divisible by $p$.

\begin{theorem}
Let $M=M(g;S)$ be a Seifert manifold and $p$ an odd prime such that 
$\mathbb{Z}_{(p)}\otimes{T(M)}$ is homogeneous of exponent $p^k$.
Then
\begin{enumerate}
\item{if} $\varepsilon_S=0$ then 
$d(\ell_M)=[(-1)^{r_p-1}\frac{\alpha_1}{\alpha_2}
(\Pi_{1\leq{i}\leq{r}}\beta_i)(\Pi_{3\leq{j}\leq{r}}u_j)]$;

\item{if} $\varepsilon_S\not=0$ then
$d(\ell_M)=
[(-1)^{r_p-1}(\Pi_{1\leq{i}\leq{r_p}}\beta_i)(\Pi_{2\leq{j}\leq{r_p}}u_j)v].$
\end{enumerate}
\end{theorem}

\begin{proof}
Suppose first that $\varepsilon_S=0$.
Then $\frac{\alpha_1}{\alpha_2}=\frac{u_1}{u_2}$ 
is also invertible in $\mathbb{Z}_{(p)}$,
and $\mathbb{Z}_{(p)}\otimes{T(M)}\cong(\mathbb{Z}/p^k\mathbb{Z})^{r_p-2}$,
with basis $e_i=q_{i+2}'$, for $1\leq{i}\leq{r_p-2}$.
We apply row operations
\[row_i\mapsto{row_i-\frac{\alpha_3\beta_{i+2}}{\alpha_{i+2}\beta_3}row_1}\] 
for $2\leq{i}\leq{n}$, and then
\[row_1\mapsto{row_1-\frac{\alpha_2\beta_3}{\alpha_3\beta_2}\Sigma_{i\geq2}row_i}.\] 
This gives a lower triangular matrix with diagonal
\[[-\beta_2\frac{\beta_3}{u_3}(\beta_2+\mu\Sigma\frac{\beta_i}{u_i}),
-\beta_2^2\frac{\beta_4}{u_4},\dots,-\beta_2^2\frac{\beta_{r_p-2}}{u_{r_p-2}}].\]
Therefore 
\[\mathrm{det}(L)=(-1)^n\beta_2^{2n-1}
(\beta_2+\mu\Sigma\frac{\beta_i}{u_i})\Pi\frac{\beta_i}{u_i},\]
and so 
\[d(\ell_M)=[(-1)^{r_p-1}\frac{\alpha_1}{\alpha_2}
(\Pi_{1\leq{i}\leq{r}}\beta_i)(\Pi_{3\leq{j}\leq{r}}u_j)].
\]

A similar argument applies if $\varepsilon_S\not=0$.
In this case $u_2$ and $v=\alpha_1\varepsilon_S$ 
are also invertible in $\mathbb{Z}_{(p)}$,
and $\mathbb{Z}_{(p)}\otimes{T(M)}\cong(\mathbb{Z}/p^k\mathbb{Z})^{r_p-1}$,
with basis $e_i=q_{i+2}'$, for $1\leq{i}\leq{r_p-2}$, and ${e_{r_p-1}=s}$.
We now have $n=r_p-1$ and the final row and column of the matrix $L$ 
have a slightly different form.
If we perform the same row operations on the first $n-1$ rows,
and then the column operation 
$col_1\mapsto{col_1+\Sigma_{2\leq{i}\leq{r_p-2}}col_i}$
we obtain a bordered matrix
\[
\left( 
\begin{matrix}
-\beta_2\frac{\beta_3}{u_3}(\beta_2+\mu\Sigma^*) & 0 &\hdots & 
0 &\frac{\beta_3}{u_3}\\
0 & -\beta_2^2\frac{\beta_4}{u_4}  &\hdots &0 &0\\
\vdots & 0 & \ddots &\vdots   \\
0&\hdots&0&-\beta_2^2\frac{\beta_{r_p}}{u_{r_p}}&0\\
\Sigma^* & \frac{\beta_4}{u_4} &\hdots  & \frac{\beta_{r_p}}{u_{r_p}}& d^*\\
\end{matrix}
\right),
\]
where $\Sigma^*=\Sigma_{3\leq{i}\leq{r_p}}\frac{\beta_i}{u_i}$
and
$d^*=-\frac{u_1+nu_2v}{u_2^2v}$.
Hence 
\[\mathrm{det}(L)=-(\beta_2\frac{\beta_3}{u_3}(\beta_2+\mu\Sigma^*)d^*+
\frac{\beta_3}{u_3}\Sigma^*)(-1)^{r_p-3}\beta_2^{2(r_p-3)}\Pi_{4\leq{i}\leq{r_p}}
\frac{\beta_i}{u_i}
\]
\[=(-\beta_2d^*(\beta_2+\mu\Sigma^*)+\Sigma^*)
(-1)^{r_p-2}\beta_2^{2(r_p-3)}\Pi_{3\leq{i}\leq{r_p}}\frac{\beta_i}{u_i}.
\]
Now 
$(-\beta_2d^*(\beta_2+\mu\Sigma^*)+\Sigma^*)=$
\[=-(\beta_2(u_1+nu_2v)(-\beta_2+u_2\Sigma^*)-u_2^2v\Sigma^*)/u_2^2v\]
\[=(\beta_2(u_1\beta_2+u_1u_2\Sigma^*+n\beta_2u_2v)+
(n\beta_2-1)u^2v\Sigma^*)/u_2^2v\]
\[\equiv\beta_2(u_1\beta_2+u_1u_2\Sigma^*+u_2v)\quad{mod}~(p)\]
\[\equiv-\beta_1\beta_2u_2\quad{mod}~(p),\]
since $n\beta_2\equiv1$ {\it mod} $(p)$ and
$v=-\beta_1-\frac{\beta_2u_1}{u_2}-u_1\Sigma^*$.
Therefore
\[\mathrm{det}(L)\equiv(-1)^{r_p-1}\beta_2^{2(r_p-3)}
(\Pi_{1\leq{j}\leq{r_p}}\beta_i)/
(\Pi_{2\leq{j}\leq{r_p}}u_j)v\quad{mod}~(p),\]
and so we now have
\[d(\ell_M)=
[(-1)^{r_p-1}(\Pi_{1\leq{i}\leq{r_p}}\beta_i)(\Pi_{2\leq{j}\leq{r_p}}u_j)v].\]
\end{proof}

When all the cone point orders have the same $p$-adic valuation
(i.e., $u_1$ and $u_2$ are also invertible in $\mathbb{Z}_{(p)}$)
then these formulae for $d(\ell_M)$ are
invariant under permutation of the indices.
For if $\varepsilon_S=0$ then $[\frac{\alpha_1}{\alpha_2}]=[u_1u_2]$
in $\mathbb{F}_p^\times/(\mathbb{F}_p^\times)^2$,
while if  $\varepsilon_S\not=0$ then
$v=u_1p^k\varepsilon_S$ (and $p^k\varepsilon_S$ is also invertible).

A linking pairing $\ell$ on a free $\mathbb{Z}/p^k\mathbb{Z}$-module $N$ 
is hyperbolic if and only if $\rho=rk(\ell)$ is even 
and $d(\ell)=[(-1)^{\frac{\rho}2}]$.
Thus $\mathbb{Z}_{(p)}\otimes\ell_M$ 
is hyperbolic if and only if either $\varepsilon_S=0$, 
$r_p=\rho+2$ is even and 
$[\frac{\alpha_1}{\alpha_2}
(\Pi_{1\leq{i}\leq{r_p}}\beta_i)(\Pi_{3\leq{j}\leq{r_p}}u_j)]
=[(-1)^{\frac{r_p}2-1}]$
or $\varepsilon_S\not=0$, $r_p=\rho+1$ is odd and
$[(\Pi_{1\leq{i}\leq{r_p}}\beta_i)(\Pi_{2\leq{j}\leq{r_p}}u_j)v]=
[(-1)^{\frac{r_p-1}2}].$

\section{realization of pairings on groups of odd order}

In this section we shall show that every linking pairing 
on a finite group of odd order may be realized by a
Seifert manifold $M(0;S)$.

Suppose first that we localize $\ell_M$ at a prime $p$.
Let $p^k$ be the exponent of $\mathbb{Z}_{(p)}\otimes{T(M)}$,
and let $L$ be the $\mathbb{Z}/p^k\mathbb{Z}$-matrix 
with entries $p^k\ell(q_i',q_j')$.
(If $\varepsilon_S\not=0$ we need also a row and column
corresponding to the generator $s$.) 
Then
\[ L=
\left(
\begin{matrix}
D_1& p^{\kappa_2}B_2 & \dots & p^{\kappa_t}B_t\\
p^{\kappa_2}B_2^{tr} & p^{\kappa_2}D_2 &\dots &\vdots\\
\vdots & \dots & \ddots & \vdots\\
p^{\kappa_t}B_t^{tr}& 0& \dots &p^{\kappa_t}D_t
\end{matrix}
\right),
\]
where $D_i$ is a $\rho_i\times\rho_1$ block with
$\mathrm{det}(D_i)\not\equiv0$ {\it mod} $(p)$,
for $1\leq{i}\leq{t}$, and $0<\kappa_2<\dots<\kappa_t<k$. 
We may partition $L$ more coarsely
as $L=\left(\smallmatrix A&B\\
B^{tr}&D\endsmallmatrix\right)$,
where $A=D_1$ and $B=[B_2\dots B_t]$.
Let 
$Q=\left(\smallmatrix I&-D_1^{-1}p^{\kappa_2}{B}\\
0&I\endsmallmatrix\right)$
and $D'=D-p^{\kappa_2}B^{tr}A^{-1}B$.
Then
$Q^{tr}LQ=\left(\smallmatrix A&0\\
0&p^{\kappa_2}{D'}\endsmallmatrix\right)$.
Block-diagonalizing $L$ in this fashion
does not change the residues {\it mod} $(p)$ of the diagonal blocks $D_i$
or decrease the $p$-divisibility of the off-diagonal blocks.
We may iterate this process, and we find that $\ell_M$
is an orthogonal sum of pairings on homogeneous groups 
$(\mathbb{Z}/p^k\mathbb{Z})^{\rho_1},
(\mathbb{Z}/p^{k-\kappa_2}\mathbb{Z})^{\rho_2},
\dots,(\mathbb{Z}/p^{k-\kappa_t}\mathbb{Z})^{\rho_t}$.
If $\varepsilon_S=0$ or if $\alpha_1\varepsilon_S$ is invertible
in $\mathbb{Z}_{(p)}$ then the determinantal invariants 
of the first summand (with the maximal exponent $p^k$) 
may be computed from the block $A$ as in \S4, 
while we may read off the determinantal invariants 
of the other summands from the corresponding diagonal elements 
of the original matrix $L$.
(We shall not need to consider the possibility that $p$ 
divides the numerator of $\varepsilon_S$ in justifying
our constructions below.)

With these reductions in mind, we may now contruct
Seifert manifolds realizing given pairings.

\begin{theorem}
Let $p$ be an odd prime, and let $\ell$ be a linking pairing 
on a finite $p$-primary abelian group.
Then there is a Seifert manifold $M=M(0;S)$ such that 
the cone point orders $\alpha_i$ are all powers of $p$,
$\varepsilon_S=0$ and $\ell_M\cong\ell$.
\end{theorem}

\begin{proof}
The pairing $\ell$ is the orthogonal sum $\perp_{j=1}^t\ell_j$,
where $\ell_j$ is a pairing on $(\mathbb{Z}/p^{k_j}\mathbb{Z})^{\rho_j}$,
with $\rho_j>0$ for $1\leq{j}\leq{t}$ and 
$0<k_j<k_{j-1}$ for $2\leq{j}\leq{t}$.
Let $d(\ell_j)=[w_j]$ for $1\leq{j}\leq{t}$, and let $k=k_1$.

If $p\geq5$ we let $\alpha_i=p^k$ for $1\leq{i}\leq{m_1=\rho_1+2}$,
$\alpha_i=p^{k_2}$ for $m_1<i\leq{m_2=m_1+\rho_2}$, $\dots$,
and $\alpha_i=p^{k_t}$ for $m_{t-1}<i\leq(\Sigma\rho_j)+2$.
For each $1\leq{j}\leq{t}$ we let $\beta_i=1$ for
$m_j<i<m_{j+1}$ and $\beta_{m_{j+1}}=w_j$.
We must then choose $\beta_i$ for $1\leq{i}\leq{m_1}$ so that
$[\Pi_{1\leq{i}\leq{m_1}}\beta_i]=[w_1]$ and 
$\Sigma_{1\leq{i}\leq{m_1}}\beta_i=
-\Sigma_{m_1<i\leq{r}}p^k\frac{\beta_i}{\alpha_i}$.
It is in fact sufficient to solve the equation
$\Sigma_{1\leq{i}\leq{m_1}}\beta_i=0$ with
all $\beta_i\in(\mathbb{Z}/p^k\mathbb{Z})^\times$,
and $[\Pi_{1\leq{i}\leq{m_1}}\beta_i]=[w_1]$,
for subtracting $\Sigma_{m_1<i\leq{r}}p^k\frac{\beta_i}{\alpha_i}$
from $\beta_1$ will not change its residue {\it mod} $(p)$.

If  $m_1$ is odd the equation $\Sigma\beta_i=0$
always has solutions with all $\beta_i\in(\mathbb{Z}/p^k\mathbb{Z})^\times$.
If $\xi$ is a nonsquare in $(\mathbb{Z}/p^k\mathbb{Z})^\times$
setting $\beta_i'=\xi\beta_i$ for all $i$ gives another solution,
and $[\Pi\beta_i']=[\xi][\Pi\beta_i]$.
(If $p\equiv3$ {\it mod} (4) we may take $\xi=-1$, 
which corresponds to a change of orientation of the 3-manifold.)

If $m_1=4t$ and $w\not\equiv1$ {\it mod} $(p)$
there is an integer $x$ such that $x\equiv\frac12(w-1)$ 
{\it mod} $(p)$.
The images of $x$ and $w-1-x$ are invertible in $\mathbb{Z}/p^k\mathbb{Z}$.
Let $\beta_1=1$, $\beta_2=-w$, $\beta_3=x$ and $\beta_4=w-1-x$
and $\beta_{2i+1}=1$ and $\beta_{2i+2}=-1$ for $2\leq{i}<2t$.
Then $\Sigma\beta_i=0$,
$\beta_4\equiv\beta_3$ {\it mod} $(p)$ and $[(-1)^{r-1}\Pi\beta_i]=[w]$.

If $m_1=4t+2$ and $w\not\equiv1$ {\it mod} $(p)$ let
$\beta_1=1$, $\beta_2=w$, $\beta_3=\beta_4=\beta_5=y$, $\beta_6=w-1-3y$ 
and $\beta_{2i+1}=1$ and $\beta_{2i+2}=-1$ for $3\leq{i}\leq2t$,
where $y\equiv-\frac14(1+w)$ {\it mod} $(p)$.
Then $\Sigma\beta_i=0$,
$\beta_6\equiv\beta_3$ {\it mod} $(p)$ and $[(-1)^{r-1}\Pi\beta_i]=[w]$.

These choices work equally well for all $p\geq3$, if $[w]\not=1$.
If $m_1$ is even, $w\equiv1$ {\it mod} $(p)$ and $p>3$ 
there is an integer $n$ such that $n^2\not=0$ or 1 {\it mod} $(p)$,
and we solve as before, after replacing $w$ by $\hat{w}=n^2w$.

However if $p=3$ and $[w]=1$ we must vary our choices.
If $m_1=4t$ with $t>1$ 
let $\beta_1=\beta_2=\beta_3=\beta_4=1$, $\beta_5=\beta_6=-2$
and $\beta_{2i+1}=1$ and $\beta_{2i+2}=-1$ for $3\leq{i}<2t$.
If $m_1=4t+2$ let
$\beta_{2i-1}=1$ and $\beta_{2i}=-1$ for $1\leq{i}\leq2t+1$.
In the remaining case (when $m_1=4$)
we find that if $\Sigma_{1\leq{i}\leq4}\beta_i=0$ then $[-\Pi\beta_i]=[-1]$.
In this case we must use instead
$S=((3^{k+1},1),(3^{k+1},5),(3^k,-1),(3^k,-1))$
to realize the pairing with ${[w]=[1]}$.
\end{proof}

The manifolds with Seifert data as above are
$\mathbb{H}^2\times\mathbb{E}^1$-manifolds,
except when $\rho=1$ and $p=3$, in which case 
they are the flat manifold $G_3$ (with its two possible orientations).
The manifold $M=M(1;(3,1),(3,1),(3,-2))$ is an
$\mathbb{H}^2\times\mathbb{E}^1$-manifold with $T(M)\cong{\mathbb{Z}/3\mathbb{Z}}$
and base orbifold a torus with cone points.

It follows immediately from Theorem 4 and Lemma 1
that every linking pairing on a finite abelian group of odd order
is realized by some Seifert manifold $M(0;S)$ with $\varepsilon_S=0$.
All such pairings may also be realized by Seifert fibred manifolds 
which are $\mathbb{Q}$-homology spheres (which have $\varepsilon_S\not=0$).
However we must be carefull to ensure that the {\it numerator} 
of $\varepsilon_S$ does not provide unexpected torsion.

\begin{theorem}
Let $\ell$ be a linking pairing on a finite abelian group $A$ of odd order.
Then there is a Seifert manifold $M=M(0;S)$ such that 
$\varepsilon_S\not=0$ and $\ell_M\cong\ell$.
\end{theorem}

\begin{proof}
Let $A=\oplus_{p\in{P}}A(p)$ and $\ell=\perp_{p\in{P}}\ell^{(p)}$ 
be the primary decompositions of $A$ and $\ell$.
For each $p\in{P}$ we shall define a Seifert data set $S(p)$ as follows.
Suppose that $\ell^{(p)}$ is a pairing on
$\oplus_{j=1}^{t(p)}(\mathbb{Z}/p^{k_j}\mathbb{Z})^{\rho_j}$,
with $\rho_j>0$ for $1\leq{j}\leq{t(p)}$
and $k_j\geq{k_{j+1}}>0$ for $1\leq{j}<{t}$.
Then $\ell^{(p)}=\perp\ell_{b_i/a_i}$,
where the $a_i$ are powers of $p$, with $a_i\geq{a_{i+1}}$
for all $i$.
Let $S(p)=((\alpha_1^{(p)},\beta_1^{(p)}),\dots,
(\alpha_{t(p)}^{(p)},\beta_{t(p)}^{(p)}))$,
where $\alpha_i^{(p)}=a_i$, $\beta_1^{(p)}=(-1)^{\rho_1+1}{b_1}$
and $\beta_i^{(p)}=b_i$, 
for $2\leq{i}\leq\rho(p)=\Sigma_{j=1}^{t(p)}\rho_j$.
Finally let $\widetilde\alpha=e\Pi_{p\in{P}}p$,
where $e$ is the exponent of $A$, and
$\widetilde\beta=-1-\widetilde\alpha
\Sigma_{p\in{P}}\Sigma_{i=1}^{t(p)}\frac{\beta_i^{(p)}}{\alpha_i^{(p)}}$.

Let $S$ be the concatenation of $\{(\widetilde\alpha,\widetilde\beta)\}$
and the $S(p)$s for $p\in{P}$, and let $M=M(0;S)$.
Then $\varepsilon_S=\frac1{\widetilde\alpha}$.
For each $p\in{P}$ there are $\rho(p)+1$ cone points with order
divisible by $p$, 
and $\mathbb{Z}_{(p)}\otimes{T(M)}\cong
\oplus_{j=1}^{t(p)}(\mathbb{Z}/p^{k_j}\mathbb{Z})^{\rho_j}$.
Since $\widetilde\beta\equiv-1$ {\it mod} $(\widetilde\alpha)$,
the determinantal invariant of the component of
$\mathbb{Z}_{(p)}\otimes\ell_M$ of maximal exponent 
$a_1=p^{k_1}$ is $[\Pi{b_i}]$.
Therefore $\mathbb{Z}_{(p)}\otimes\ell_M\cong\ell^{(p)}$,
for each $p\in{P}$,
and so $\ell_M\cong\ell$.
\end{proof}

The manifolds with Seifert data as above are
$\widetilde{\mathbb{SL}}$-manifolds,
except when $A\cong\mathbb{Z}/p^k\mathbb{Z}$ 
and so there are just two cone points,
in which case they are lens spaces ($\mathbb{S}^3$-manifolds).

In the homogeneous $p$-primary case we may 
arrange that all cone points have order $p^k$,
except when $p=3$, $\rho=2$ and $d(\ell)=[1]$.
(This case is realized by $M(0;(3^{k+1},7),(3^k,-1),(3^k,-1))$.)

\section{realization of homogeneous $2$-primary pairings}

The situation is more complicated when $p=2$.
A linking pairing $\ell$ on $(\mathbb{Z}/2^k\mathbb{Z})^\rho$ 
is determined by its rank $\rho$ and certain invariants 
$\sigma_{j}(\ell)\in{\mathbb{Z}/8\mathbb{Z}}\cup\{\infty\}$, 
for $\rho-2\leq{j}\leq\rho$.
(See \S3 of \cite{KK}, and \cite{De}.)
We shall not calculate these invariants here.
Instead, we shall take advantage of the particular form 
of the pairings given in \S3. 

If a linking pairing $\ell$ on $(\mathbb{Z}/2^k\mathbb{Z})^\rho$ 
is even then $\rho$ is also even, 
and either $\ell$ is hyperbolic 
(and is the orthogonal sum of $\frac{\rho}2$ copies of the pairing $E^k_0$) 
or it is the orthogonal direct sum of a hyperbolic pairing of rank $\rho-2$
with the pairing $E_1^k$ (if $k>1$) \cite{KK,Wa64}.
When $k=1$ all even pairings are hyperbolic.
Otherwise, $\ell$ is determined by the image of the matrix
$L=2^k\ell(e_i,e_j)$ in $GL(\rho,\mathbb{Z}/4\mathbb{Z})$.
In particular, if $k>1$ and $L=\left(\smallmatrix 0& c\\
c&d\endsmallmatrix\right)$ with $c$ odd and $d$ even the
pairing is hyperbolic.

We shall say that an element $\frac{p}q\in\mathbb{Z}_{(2)}$ 
is {\it even} or {\it odd} if $p$ is even or odd, respectively.
(Thus $\frac{p}q$ is odd if and only if it is invertible in $\mathbb{Z}_{(2)}$.)

The following result complements the criterion for homogeneity given 
in Lemma 2.

\begin{lemma}
Let $M=M(g;S)$ be a Seifert manifold and 
let $\ell=\mathbb{Z}_{(2)}\otimes\ell_M$.
Assume that the Seifert data are ordered so that
$\alpha_{i+1}$ divides $\alpha_i$ in $\mathbb{Z}_{(2)}$.
Then 
\begin{enumerate}
\item $\ell$ is even if and only if $\frac{\alpha_1}{\alpha_i}$ is odd
for $1\leq{i}\leq{r_2}$ and either $\varepsilon_S=0$ or 
$\alpha_1\varepsilon_S$ is odd;
\item{if} $\frac{\alpha_1}{\alpha_i}$ is odd for $1\leq{i}\leq{r_2}$
then $\alpha_1\varepsilon_S\equiv{r_2}$ mod $(2)$.
\end{enumerate}
\end{lemma}

\begin{proof}
If $\mathbb{Z}_{(2)}\otimes\ell_M$ is even then 
$\beta_2+\frac{\alpha_2}{\alpha_i}\beta_i$ is even for all $3\leq{i}\leq{r_2}$.
Hence $\frac{\alpha_2}{\alpha_i}$ is odd,
since the $\beta_i$ are all odd.
If moreover $\varepsilon_S=0$ then $\frac{\alpha_1}{\alpha_2}$ is odd.
If $\varepsilon_S\not=0$ then 
$\frac{\alpha_1}{\alpha_2}+n\alpha_1\varepsilon_S$ is even.
Hence $\frac{\alpha_1}{\alpha_2}$ must again be odd,
and so $\alpha_1\varepsilon_S$ is also odd.
In each case, the converse is clear.

The second assertion holds since $\beta_i$ is odd for $1\leq{i}\leq{r_2}$
and $\frac{\alpha_i}{\alpha_i}$ is even for all $i>r_2$.
\end{proof}

We shall suppose for the remainder of this section that 
$\mathbb{Z}_{(2)}\otimes{T(M)}$ is homogeneous of exponent $2^k>1$.

Suppose first that $\mathbb{Z}_{(2)}\otimes\ell_M$ is even.
Then it is homogeneous and of even rank $\rho=2s$.
The diagonal entries of $L$ are all even 
and the off-diagonal entries are all odd.
If $k=1$ then $\mathbb{Z}_{(2)}\otimes\ell_M$ is hyperbolic,
so we may assume that $k>1$.

\begin{theorem}
Let $M=M(g;S)$ be a Seifert manifold such that the even cone
point orders $\alpha_i$ all have the same $2$-adic valuation $k>1$.
Assume that either $\varepsilon_S=0$ or $\alpha_1\varepsilon_S$ is odd.
Let $t$ be the number of diagonal entries of $L$ which are divisible by $4$.
Then  whether $\mathbb{Z}_{(2)}\otimes\ell_M$ is hyperbolic or not
depends only on the images of $t$ and $\rho$ in $\mathbb{Z}/4\mathbb{Z}$.
\end{theorem}

\begin{proof}
The linking pairing is even, by Lemma 6, and so $\rho$ is even.
We may reorder the basis of $T(M)$ so that $L_{ii}\equiv0$ {\it mod} (4),
for all $i\leq{t}$ and $L_{ii}\equiv2$ {\it mod} (4) for $t<i\leq\rho$.
Let $t=4a+x$ and $\rho-t=4b+y$, where $0\leq{x,y}\leq3$. 
Then $E=\left(\smallmatrix L_{11}&L_{12}\\
L_{21}&L_{22}\endsmallmatrix\right)$ is invertible.
We may partition $L$ as $L=\left(\smallmatrix E& F\\
F^{tr}&G\endsmallmatrix\right)$,
where $G$ is a $(\rho-2)\times(\rho-2)$ submatrix
and $F$ is a $2\times(\rho-2)$ submatrix.
If we conjugate by $J=\left(\smallmatrix I_2&-E^{-1}F\\
0&I_{\rho-2}\endsmallmatrix\right)$ to obtain
$J^{tr}LJ=\left(\smallmatrix{E}&0\\
0&G'\endsmallmatrix\right)$,
then $G'=G-F^{tr}E^{-1}F$.
The entries of $F$ are all odd and so the entries of 
$F^{tr}E^{-1}F$ are all congruent to 2 {\it mod } (4).
Therefore $G'\equiv{G}$ {\it mod} (2),
and so the off-diagonal entries of $G'$ are still odd,
but the residues {\it mod} (4) of the diagonal entries are changed.
An application of this process to $G'$ then restores the
residue classes of the diagonal entries of the corresponding 
$(\rho-4)\times(\rho-4)$-submodule.
Iterating this process, we find that 
$\ell_M\cong (a+b)(E_0^k\perp{E_1^k})\perp\ell'$,
where $\ell'$ has rank $x+y$ and the off-diagonal entries
of the matrix for $\ell'$ are odd.
We also find that $\ell'\cong(x+y)E_0$,
unless $(x,y)=(1,3)$ or (0,2),
in which case $\ell'\cong{E_0^k}\perp{E_1^k}$ or $E_1^k$,
respectively.
Since $2E_1^k\cong{2E_0^k}$ \cite{Wa64},
and the pairing with matrix congruent to
$\left(\smallmatrix 0&1\\
1&2\endsmallmatrix\right)$ {\it mod} (4) is hyperbolic,
it follows that $\ell$ is hyperbolic if and only if either $a+b$ is even and 
$(x,y)\not=(1,3)$ or (0,2), or if $a+b$ is odd and
$(x,y)=(1,3)$ or (0,2).
\end{proof}

It follows immediately from the calculations in \S3 that
\[t=\#\{i\geq3\mid\frac{\alpha_2\beta_i+\alpha_i\beta_2}{2^k}
\equiv0~mod~(4)\}+\delta,\]
where $\delta=1$ if $\varepsilon_S\not=0$ and 
$\beta_2+\alpha_2\varepsilon_S\equiv0$ mod $(4)$,
and $\delta=0$ otherwise.

In particular, 
if $t$ and $\rho$ are divisible by 4 then $\ell_M$ is hyperbolic.

Lemmas 2 and 6 also imply that if $\mathbb{Z}_{(2)}\otimes{T(M)}$ 
is homogeneous of exponent $2^k$ then $\mathbb{Z}_{(2)}\otimes\ell_M$ 
is odd if and only if either 
\begin{enumerate}
\item $\varepsilon_S=0$,
$2^{-k}\alpha_i$ is even for $i=1$ and 2,
and is odd for $2<i\leq{r_2}$; or
\item $\alpha_1\varepsilon_S$ and $2^{-k}\alpha_i$ are odd 
for $1<i\leq{r_2}$, 
and either $2^{-k}\alpha_1$ or $r_2$ is even.
\end{enumerate}
Odd forms on homogeneous 2-groups can be diagonalized.
In the present situation, this follows easily from the next lemma.

\begin{lemma}
Let $\ell$ be an odd linking pairing on $N=(\mathbb{Z}/2^k\mathbb{Z})^2$.
Then $\ell$ is diagonizable.
\end{lemma}

\begin{proof}
Let $e,f$ be the standard basis for $N$.
Since $\ell$ is odd we may assume that $\ell(e,e)=[2^{-k}a]$, 
where $a$ is odd.
Let $\ell(e,f)=[2^{-k}b]$ and $\ell(f,f)=[2^{-k}d]$.  
(Then $b$ is even and $d$ is odd, or vice-versa,
by nonsingularity of the pairing.)
Let $f'=-a^{-1}be+f$.
Then $\ell(e,f')=0$ and $\ell(f',f')=[2^{-k}d']$,
where $d'\equiv{d-a^{-1}b^2}$ {\it mod} $(2^k)$.
Therefore $\ell\cong\ell_{\frac{a}{2^k}}\perp\ell_{\frac{d'}{2^k}}$.
\end{proof}

Note that if $b\equiv0$ {\it mod} (4) then
$\ell_{\frac{d'}{2^k}}\cong\ell_{\frac{d}{2^k}}$.

Suppose now that $\mathbb{Z}_{(2)}\otimes\ell_M$ is odd and 
$\varepsilon_S=0$.
Then the diagonal entries of $L$ are odd 
and the off-diagonal elements are odd multiples of $2^{-k}\alpha_1$.
We may assume also that $r_2\geq4$, for otherwise 
$\mathbb{Z}_{(2)}\otimes{T(M)}$ is cyclic.
We may partition $L$ as 
$L=\left(\smallmatrix E& F\\
F^{tr}&G\endsmallmatrix\right)$,
where $E\in{GL(2,\mathbb{Z}_{(2)})}$,
$F$ is a $2\times(r_2-4)$ submatrix with even entries
and $G$ is a $(r_2-4)\times(r_2-4)$ submatrix.
Let $J=\left(\smallmatrix I_2&-E^{-1}F\\
0&I_{r_2-4}\endsmallmatrix\right)$.
Then $\mathrm{det}(J)=1$ and 
$J^{tr}LJ=\left(\smallmatrix E&0\\
0&G'\endsmallmatrix\right)$,
where $G'=G-F^{tr}E^{-1}F$.
The columns of $F$ are proportional, and the ratio $u_3/u_4$ is odd.
Since the entries of $F$ are odd multiples of $2^{-k}\alpha_1$
and since $E-I_2$ has even entries, 
$G'\equiv{G}$ {\it mod} (8).
Iterating this process, we may replace $L$ by a block-diagonal matrix,
where the blocks are all $2\times2$ or $1\times1$,
and are congruent {\it mod} (8) to the corresponding blocks of $L$.
Each such $2\times2$ block is diagonalizable, by Lemma 8,
and so we may easily represent $\ell_M$ as an
orthogonal sum of pairings of rank 1.

If $\varepsilon_S\not=0$ then $\ell(q_i',s)=[2^{-k}\frac{\beta_i}{u_i}]$
and $\ell_M(s,s)=[2^{-k}z]$, 
where $\beta_i$, $u_i$ and $z$ are odd,
and we first replace each $q_i'$ by 
$\tilde{q}_i=q_i'-z^{-1}u_i^{-1}\beta_is$.
We then see that $\mathbb{Z}_{(2)}\otimes\ell_M\cong\ell_{2^{-k}z}\perp\tilde\ell$,
where the matrix for $\tilde\ell$ has odd diagonal entries
and even off-diagonal entries, and we may continue as before.

\begin{theorem}
Let $\ell$ be a linking pairing on $(\mathbb{Z}/2^k\mathbb{Z})^\rho$.
Then there is a Seifert manifold $M=M(0;S)$ such that the cone point 
orders $\alpha_i$ are all powers of $2$ and $\ell_M\cong\ell$.
We may have either $\varepsilon_S=0$ or $\varepsilon_S\not=0$.
\end{theorem}

\begin{proof}
Suppose first that $\ell$ is even.
Then $\rho$ is also even, and
$\ell_M\cong(E_0^k)^{\frac{\rho}2}$ or
$(E_0^k)^{\frac{\rho}2-1}\perp{E_1^k}$.
Let $S=((2^k,\beta_1),\dots,(2^k,\beta_r))$
with $\beta_i=(-1)^i$ for ${1\leq{i}\leq{r}}$.
Then $\varepsilon_S=-\frac1{2^k}$ if $r$ is odd and
$\varepsilon_S=0$ if $r$ is even, 
and $\ell_M\cong(E_0^k)^{\frac{\rho}2}$. 

If $\rho\equiv2$ {\it mod} (4) let $\beta_1=-3$, 
$\beta_2=\beta_3=1$ and $\beta_i=(-1)^i$ for $4\leq{i}\leq{r}$,
where $r=\rho+1$ or $\rho+2$.
If $\rho\equiv0$ {\it mod} (4) let $\beta_1=-5$, 
$\beta_2=\dots=\beta_5=1$ and $\beta_i=(-1)^i$ for $6\leq{i}\leq{r}$,
where $r=\rho+1$ or $\rho+2$.
In each case $\varepsilon_S=-\frac1{2^k}$ if $r$ is odd and
$\varepsilon_S=0$ if $r$ is even, 
and $\ell_M\cong(E_0^k)^{\frac{\rho}2-1}\perp{E_1^k}$. 

Now suppose that $\ell$ is odd.
Then $\ell\cong\perp_{i=1}^\rho\ell_{w_i}$,
where $w_i=2^{-k}b_i$ for $1\leq{i}\leq\rho$.
If $\varepsilon_S=0$ we let $r=\rho+2$ and
$S=((\alpha_1,\beta_1),\dots,(\alpha_r,\beta_r))$,
where $\alpha_1=\alpha_2=2^{k+2}$, 
$\alpha_i=2^k$ for $3\leq{i}\leq{r}$, $\beta_2=1$, 
$\beta_i=3b_i$ for $3\leq{i}\leq{r}$ and 
$\beta_1=-1-4\Sigma_{i\geq3}\beta_i$.
Then $\ell_M\cong\ell$.
(Here we may use the observation following Lemma 8.)

If $\varepsilon_S\not=0$ we let $r=\rho+1$.
Here we must take into account the change of basis suggested 
in the paragraph before the theorem.
Suppose first that some $b_i\equiv\pm3$ {\it mod} (8).
We may then arrange that $z=3$ and the matrix for $\tilde\ell$ 
is congruent {\it mod} (4) to a diagonal matrix.
After reordering the summands, 
and allowing for a change of orientation,
we may assume that $b_1\equiv3$ {\it mod} (8).
Then we let $S=((\alpha_1,\beta_1),\dots,(\alpha_r,\beta_r))$,
where $\alpha_1=2^{k+2}$, 
$\alpha_i=2^k$ for $2\leq{i}\leq{r}$, $\beta_2=1$, 
$\beta_i=4-b_i$ for $3\leq{i}\leq{r}$ and $\beta_1=1-4\Sigma_{i=2}^r\beta_i$.
Then $\varepsilon_S=-\frac1{2^k}$ and $\ell_M\cong\ell$.

Finally, suppose that $b_i\equiv\pm1$ {\it mod} (8) for all $i$.
If $\rho=1$ let $S=((2^{k+1},1),(2^k,1))$.
If $\rho=2$ and $b_1\equiv-b_2\equiv1$ {\it mod} (8),
let $S=((2^{k+1},1),(2^k,1),(2^k,-1)$.
Otherwise we may assume that $b_1\equiv{b_2}$ {\it mod} (8).
But then $\ell_{w_1}\perp\ell_{w_2}\cong2\ell_{w'}$,
where $w'=2^{-k}.3$ \cite{KK},
and we are done.
\end{proof}

\begin{cor}
Let $\ell$ be a linking pairing on a finite abelian group 
with homogeneous $2$-primary subgroup.
Then there is a Seifert manifold $M=M(0;S)$ 
such that $\ell_M\cong\ell$.
We may have either $\varepsilon_S=0$ or $\varepsilon_S\not=0$.
\end{cor}

\begin{proof} 
The case with $\varepsilon_S=0$ follows immediately from 
Theorems 4 and 9 and Lemma 1.
If $\varepsilon_S\not=0$ we must argue as in Theorem 5.
\end{proof}

The manifolds constructed in this section are either
$\mathbb{H}^2\times\mathbb{E}^1$-manifolds (if $\varepsilon_S=0$),
or $\widetilde{\mathbb{SL}}$-manifolds (if $\varepsilon_S\not=0$),
with the exceptions of the half-turn flat manifold
$M(0;(2,-1),(2,1),(2,-1),(2,1))$
and the $\mathbb{S}^3$-manifolds $M(0;(2,1),(2,1),(2,\beta))$ 
and $M(0;(4,1),(2,1),(2,\beta))$.

\section{the inhomogeneous case: $p=2$}

The next theorem suffices to show that there are
pairings which cannot be realized by orientable Seifert fibred 3-manifolds.
(These are counter-examples to a conjecture raised in \cite{BD}.)

\begin{theorem}
Let $M=M(g;S)$ be a Seifert manifold,
and let $S(2)=((\alpha_1,\beta_1),\dots,(\alpha_t,\beta_t))$ 
be the subset of the Seifert data with even cone point orders $\alpha_i$,
in descending order of divisibility by $2$.
Then ${\mathbb{Z}_{(2)}\otimes\ell_M}$ has a nontrivial even 
component if and only if $\alpha_1, \alpha_2$ and $\alpha_3$
have the same $2$-adic valuation and $\alpha_1\varepsilon_S=0$ 
or is oddn.
The even component has exponent $\alpha_1$, 
and all the other components are diagonalizable.
\end{theorem}

\begin{proof}
The block-diagonalization process of \S5 does not change 
the parity of the entries in the diagonal blocks $D_i$.
The first assertion then follows easily from the calculations of \S3.
(In particular, note that if $\alpha_1$ and $ \alpha_2$ 
have the same $2$-adic valuation but $\alpha_3$ properly 
divides $\alpha_2$ in $\mathbb{Z}_{(2)}$ then $\alpha_1\varepsilon_S$ is even,
while if $\alpha_i$ and $\alpha_j$ properly divide $\alpha_2$ 
and $i\not=j$ then $\alpha_i\ell_M(q_i',q_i')$ is odd and 
$\alpha_i\ell_M(q_i',q_j')$ is even.)
\end{proof}

In particular, if $\alpha_1\varepsilon_S$ is even but nonzero then 
the image of $s$ in $\mathbb{Z}_{(2)}\otimes{T(M)}$
has order greater than $\alpha_1$,
and the component of highest exponent is cyclic.

For example, $\ell_{\frac14}\perp{E_0^1}$ 
is the linking pairing of the $\mathbb{N}il^3$-manifold
$M(0;(2,1),(2,1),(2,1),(2,-1))$,
but is not realizable by a Seifert manifold with $\varepsilon_S=0$.
The pairing $E_0^2\perp{E_0^1}$ is not realized by any Seifert manifold
(as defined above, i.e., with orientable base orbifold).

Let $M=M(-k;S)$ be the orientable 3-manifold which is Seifert fibred 
over a {\it non}-orientable base orbifold with underlying surface
$\#^kRP^2$, for some $k\geq1$.
If $S=((2,1),(2,1),(2,1),(2,1))$ and $k=2$ then
$\ell_M\cong{E_0^2}\perp{E_0^1}$.
However if $\mathbb{Z}_{(2)}\otimes{T(M)}$ has exponent divisible by 16 
and a direct summand of order $2$ then
there are cone point orders $\alpha_1$ and $\alpha_m$ such that
$\alpha_1$ is divisible by $4$ and $\alpha_m=2u_m$ with $u_m$ odd,
by Lemma 3.4 of \cite{CH}.
It then follows from Theorem 3.7 of \cite {CH} that
$\mathbb{Z}_{(2)}\otimes\ell_M\cong\ell'\perp\ell_{\frac12}$,
for some pairing $\ell'$.
In particular, 
$E_0^4\perp{E_0^1}$ is not realized by any orientable Seifert fibred 3-manifold
at all.

On the other hand every linking pairing $\ell$ 
on a finite abelian group is realized by some oriented 3-manifold.
(In \cite{KK} the manifold is constructed by connect-summing simple pieces,
but such a manifold may also be obtained by surgery on a framed link,
with linking matrix determined by $\ell$.
See \cite{AHV} or Exercise 5.3.13(g) of \cite{GS}.)
Since every closed orientable 3-manifold is homology cobordant 
to a hyperbolic 3-manifold \cite{My},
$\ell$ is also the linking pairing of an aspherical 3-manifold.

We expect that the conditions of Theorem 10 are the only constraint on the 
class of linking pairings realized by Seifert manifolds.
However we have only been able to confirm this under a ``gap" condition on
the 2-primary torsion.

\begin{theorem}
Let $\ell=\perp_{j=1}^t\ell_j$ be a linking pairing on a finite 
abelian $2$-group,
where $\ell_j$ is a pairing on $(\mathbb{Z}/2^{k_j}\mathbb{Z})^{\rho_j}$,
with $\rho_j>0$ for $1\leq{j}\leq{t}$, $k_j\geq{k_{j+1}+2}$ 
for $1\leq{j}<{t}$ and $k_t>0$.
Suppose also that $\ell_j$ is odd for all $2\leq{j}\leq{t}$.
Then there is a Seifert manifold $M=M(0;S)$ such that the cone point orders 
$\alpha_i$ are all powers of $2$ and $\ell_M\cong\ell$.
We may have either $\varepsilon_S=0$ or $\varepsilon_S\not=0$.
\end{theorem}

\begin{proof}
We may assume that $t>1$, since the homogeneous case is covered by Theorem 9.
Let $k=k_1$ be the exponent of $T(M)$.

In order to realize $\ell$ by a Seifert manifold with $\varepsilon_S=0$,
we must have two cone points of order at least $2^k$, 
and $\rho_j$ cone points of order $2^{k_j}$, for $1\leq{j}\leq{t}$.
If we set $\beta_2=1$ and choose the $\beta_i$ with $i>2$ compatibly with
$\ell$, then we may set $\beta_1=-2^k\Sigma_{i\geq2}\frac{\beta_i}{\alpha_i}$.

If $\ell_1$ is even then $\rho_1 $ is even
and we must have $\alpha_1=\alpha_2=2^k$ also.
If $\ell_1$ is hyperbolic let $\beta_i=(-1)^i$ for $1\leq{i}\leq\rho_1+2$;
if $\ell_1$ is even but not hyperbolic and $\rho_1\equiv2$ {\it mod} (4) 
let $\beta_3=1$ and $\beta_i=(-1)^i$ for $4\leq{i}\leq\rho_1+2$;
and if $\ell_1$ is even but not hyperbolic and $\rho_1\equiv0$ {\it mod} (4)
$\beta_3=\beta_4=\beta_5=1$ and $\beta_i=(-1)^i$ for $6\leq{i}\leq\rho_1+2$.

When $p=2$ the block-diagonalization process of \S5 
does not change the parity of the
entries of the blocks $B_m$ and $D_n$.
In our situation this process allows a more refined reduction. 
The block $A=D_1$ has diagonal entries ${-\beta_i(1+\beta_i)}$
and off-diagonal entries $-\beta_i\beta_j$, for $3\leq{i,j}\leq\rho_1+2$. 
Therefore 
$[B_m^{tr}A^{-1}B_n]_{pq}=-\beta_p\beta_q((\Sigma\beta_i^2)^2+\Sigma\beta_i^3)$,
which is even (for any $3\leq{p,q}\leq\rho$).
Hence the first step of the reduction does not change the 
image of the complementary blocks {\it mod} (4),
and the change {\it mod} (8) depends only on $\rho_1$,
since $-4\equiv4$ {\it mod} (8).
In particular,
if $k_1-k_2\geq2$ this step does not change the images {\it mod} (8)
at all.

The subsequent blocks have odd diagonal elements 
and even off-diagonal elements.
Nevertheless a similar reduction applies,
and further changes depend only on the ranks of the
homogeneous terms, provided that $k_j-k_{j+1}\geq2$ for all $j$.
It is clear that (knowing these ranks) we may then choose the
numerators $\beta_i$ to realize all the pairings $\ell_j$.

If $\ell_1$ is odd then 
$\ell_1\cong\perp_{i=1}^{\rho_1}\ell_{w_i}$,
where $w_i=2^{-k}b_i$ for $1\leq{i}\leq\rho$.
Let $\alpha_1=\alpha_2=2^{k+2}$, 
and $\beta_i=3b_i$ for $3\leq{i}\leq\rho_1+2$.
Then we may continue as before.

A similar argument applies when $\varepsilon_S\not=0$.
\end{proof}

\section{Witt classes}

Two pairings $\ell$ and $\ell'$ are {\it Witt equivalent} 
if there are metabolic pairings $\mu$ and $\mu'$ such that
$\ell\perp\mu\cong\ell'\perp\mu'$.
The set of Witt equivalence classes is an abelian group
$W(\mathbb{Q}/\mathbb{Z})$ with respect to orthogonal sum of pairings.
The canonical decomposition into primary summands
gives an isomorphism
\[W(\mathbb{Q}/\mathbb{Z})\cong\oplus_{p~prime}W(\mathbb{F}_p)
\]
where $W(\mathbb{F}_2)=\mathbb{Z}/2\mathbb{Z}$, 
$W(\mathbb{F}_p)\cong\mathbb{Z}/4\mathbb{Z}$ if $p\equiv3$ {\it mod} (4)
and $W(\mathbb{F}_p)\cong(\mathbb{Z}/2\mathbb{Z})^2$ if $p\equiv1$ {\it mod} (4).
If $a,b$ are relatively prime nonzero integers let $w(\frac{b}a)$
be the Witt class of the pairing $\ell_{\frac{b}a}$.
The summands are generated by the classes of such pairings.

In \cite{Oh} bordism arguments are used to compute
the Witt class of $\ell_M$, for $M=M(0;S)$ a Seifert manifold.
If $\varepsilon_S=0$
the Witt class of $\ell_M$ is $-\Sigma{w(\frac{\beta_i}{\alpha_i})}$,
while if $\varepsilon_S=\frac{p}q$ (in lowest form)
then it is $-w(\frac1{pq})-\Sigma{w(\frac{\beta_i}{\alpha_i})}$ \cite{Oh}.
In particular, the image of $\ell_M$ in $W(\mathbb{F}_p)$ is nontrivial if
$\varepsilon_S=0$ and $r_p$ is odd or if 
$\varepsilon_S\not=0$ and $r_p$ is even.

{\it Remark.}
The definition of Witt equivalence given here is
appropriate for obtaining bordism invariants, as in \cite{AHV,Oh}.
However the Witt groups defined in \cite{KK} use a finer equivalence 
relation, involving stabilization by split pairings (rather than by
metabolic pairings).


\end{document}